\documentclass[12pt,reqno]{amsart}
\usepackage{amssymb}
\usepackage{amsmath}

\usepackage{amscd}

\newcommand{\RNum}[1]{\uppercase\expandafter{\romannumeral #1\relax}}

\usepackage[T2A]{fontenc}
\usepackage[utf8]{inputenc}
\usepackage[russian,english]{babel}
\input{int.def}

\usepackage{tikz-cd}
\usetikzlibrary{cd}
\usepackage{dirtytalk}
\usepackage{hyperref}
\usepackage{xcolor}
\usepackage{centernot}

\usepackage{enumitem}

\numberwithin{equation}{section}

\newcommand{\Poincare}{\text{Poincar\'e }}

\usepackage[mathcal]{euscript}

\usepackage{titlesec}
\titleformat{\section}[runin]{\bfseries}{\thesection.}{3pt}{}[.]

\myitemmargin 
\usepackage{geometry}
\newgeometry{vmargin={25mm}, hmargin={18mm,18mm}, footskip=10mm}   

\begin{document}

\title[On Lagrangian tori in K3 surfaces]%
{On Lagrangian tori in K3 surfaces}

\author{Gleb Smirnov}




\begin{abstract}
Every Maslov-zero Lagrangian torus in a K3 surface has non-trivial homology class. 
This note aims to extend this result to 
Lagrangian tori with Maslov indices congruent to zero modulo 4. Conversely, we show 
that every homologically non-trivial Lagrangian torus is necessarily Maslov-zero.
\end{abstract}

\maketitle

\setcounter{section}{0}
\section{Main result}\label{main}
Let \((X, \omega)\) be a smooth symplectic 4-manifold that is homotopy-equivalent to a complex K3 surface, with the condition that \(c_1(X) = 0\). Let $L \subset X$ be an embedded Lagrangian torus. Choose an $\omega$-compatible almost-complex structure $J$ on $(X,\omega)$. Since $L$ is totally real, we have the following natural isomorphisms:
\begin{equation}\label{iso}
T_{X}|_{L} \cong T_{L} \otimes \cc,\quad 
\Lambda^2_{\cc} (T_{X}) |_{L} \cong \Lambda^2_{\rr} (T_{L}) \otimes \cc,
\end{equation}
and the latter bundle carries a canonical section since $L$ is orientable. On the other hand, $c_1(X) = 0$ implies that $\Lambda^2_{\cc} (T_{X})$ has a global section, which is unique up to homotopy since $H^1(X;\zz) = 0$. We thus obtain 
two nowhere vanishing sections of $\Lambda^2_{\cc} (T_{X})$ over $L$. They differ by a gauge map $L \to \cc^{*}$, whose class in $H^1(L;\zz)$ we denote by $\alpha$. The class $\mu = 2 \alpha$ is 
referred to as the Maslov class of $L$. For the sake of our discussion, it is more convenient to work with the class $\alpha$. If $\alpha$ vanishes, then $L$ is called a Maslov-zero torus. 
\smallskip%

If $(X,\omega)$ is a K{\"a}hler K3 surface, then every Maslov-zero Lagrangian 
torus $L \subset (X,\omega)$ has non-trivial homology class. Sheridan and Smith (see Theorem 1.4 in \cite{Sher-Smith}) proved this result for specific K{\"a}hler symplectic forms on $X$, using advanced methods of homological mirror symmetry. The theorem's generalization to arbitrary K{\"a}hler forms is due to Entov and Verbitsky, as found in Theorem 1.1 in \cite{Ent-Verbit-2}. Their proof reduces the general case to one where the Sheridan-Smith theorem is applicable. This note introduces further generalizations of these results. Specifically, we prove the following: 
\begin{theorem}\label{t:A}
Let $(X,\omega)$ be a symplectic 
4-manifold homotopy-equivalent to a complex K3 surface,  and suppose $c_1(X) = 0$. Let $L \subset (X,\omega)$ be an embedded Lagrangian torus. Then:
\begin{enumerate}[label={(\arabic*)}]
    \item If $L$ is homologically non-trivial, then $\alpha$ is zero.   
\smallskip%

    \item If $L$ is nullhomologous modulo 2, then $\alpha$ 
    is non-zero modulo 2. 
\end{enumerate}
Consequently, if the homology class of $L$ is non-trivial, its modulo 2 reduction is also non-trivial.
\end{theorem}
One can replace the assumptions that \((X,\omega)\) is homotopy-equivalent to a K3 surface and \(c_1(X) = 0\) with the single assumption that \(X\) is diffeomorphic to a K3 surface. Indeed, if \( (X, \omega) \) is a symplectic manifold diffeomorphic to a K3 surface, then \( c_1(X) = 0 \), regardless of the choice of \( \omega \). This is explained at the end of \S\,\ref{input}.
\smallskip%

Our proof builds on Nemirovski's work \cite{Nemir-1}, and it is independent of the work \cite{Sher-Smith}. In \S\,\ref{klein}, we extend this to Lagrangian Klein bottles under the same assumptions on $(X,\omega)$.

\statebf Acknowledgements. 
\ I thank Stefan Nemirovski for his comments on the manuscript. I am also grateful to the referee for their meticulous review, valuable feedback, and numerous suggestions, which have significantly improved this paper.

\section{Insights from Seiberg-Witten Theory}\label{input} 
We leverage two results from Seiberg-Witten theory for Theorem \ref{iso}. We briefly introduce Seiberg-Witten invariants, referring readers to \cite{Nic, wild, K-M} for details on four-dimensional gauge theory.
\smallskip%

Let $(X,g)$ be a smooth, closed 4-manifold with a Riemannian metric $g$, $b_1(X) = 0$, and 
$b_{2}^{+}(X) > 1$. We choose a spin$^\cc$ structure $\fr{s}$ on $X$, which gives rise to spinor bundles $W^\pm$ and determinant line bundle $\mathcal{L}$. Let \(\mathcal{A}\) denote the space of \(\UU(1)\)-connections on \(\mathcal{L}\), and let \(\mathcal{U}\) be the gauge group associated with \(\mathcal{L}\). 
\smallskip%

\(\mathcal{U}\) consists of smooth maps from \(X\) to \(\UU(1)\), forming a group since the target space is a group. When \(H^1(X; \mathbb{R}) = 0\), as is the case in this paper, each map \(u \in \mathcal{U}\) can be written as \(u = e^{-if}\) for some smooth function \(f \colon X \to \mathbb{R}\). The group \(\mathcal{U}\) acts on \(\Gamma(W^+) \times \mathcal{A}\) as follows: for \(u = e^{-if} \in \mathcal{U}\) and \((\varphi, A) \in \Gamma(W^+) \times \mathcal{A}\),  
\[
u \cdot (\varphi, A) = (e^{-if} \varphi, A + 2i \, df).
\]
If two connections \(A_1, A_2 \in \mathcal{A}\) are gauge equivalent, they differ by a closed 1-form, and so share the same curvature, i.e., \(F_{A_1} = F_{A_2}\). Since the curvature \(F_A\) is gauge-invariant (i.e., \(F_A = F_{u \cdot A}\) for all \(u \in \mathcal{U}\)), it follows that the self-dual part of the curvature, \(F_A^+\), is also gauge-invariant. This property is essential for the gauge-invariance of the Seiberg-Witten equations. 
\smallskip%

On the other hand, if \(H^1(X; \mathbb{R}) = 0\) and \(F_{A_1} = F_{A_2}\), then the difference \(A_1 - A_2\) is not only closed but also exact, meaning it can be expressed as \(2i \, df\). In this case, the gauge transformation \(e^{-if}\) will map \(A_1\) to \(A_2\). Thus, two connections are gauge-equivalent if and only if they have the same curvature.
\smallskip%

Let $\eta$ be a $g$-self-dual form on $X$. The Seiberg-Witten equations with perturbing term $\eta$ seek solutions $(\varphi, A)$ and are given by:
\begin{equation}\label{eq:sw}
 \begin{cases}
   \cald_{A} \varphi = 0, 
   \\
   F^{+}_{A}  = \sigma(\varphi) + i\,\eta,
 \end{cases}
\end{equation}
where $\cald_{A} \colon \Gamma(W^{+}) \to \Gamma(W^{-})$ is the Dirac operator, 
$\sigma(\varphi)$ is the squaring map, and $F^{+}_{A}$ 
is the self-dual part of the curvature of 
$A$. The moduli space of solutions, $M_{\fr{s}}$, is defined as:
\begin{multline}
M_{\fr{s}} = \left\{ (\varphi,A) \in \Gamma(W^{+}) \times \cala\ |\ 
\text{$(\varphi,A)$ is a solution to \eqref{eq:sw}}
\right\}/\sim,\\
\quad \text{$(\varphi,A) \sim (\varphi',A')$ if $u \cdot (\varphi',A') = (\varphi,A)$ for some $u \in \calu$.} 
\end{multline}
The moduli space $M_{\fr{s}}$ depends on both the metric $g$ and the form $\eta$. As shown in Corollary 3 in \cite{K-M}, $M_{\fr{s}}$ is always compact. For sufficiently generic choices of $g$ and $\eta$ (shown in Lemma 5 in \cite{K-M}), $M_{\fr{s}}$ becomes a smooth manifold of dimension:
\begin{equation}\label{dim-1}
d(\fr{s}) = \frac{1}{4}( c_1^2(\mathfrak{s}) - 
3 \operatorname{sign}(X) - 2 \chi(X) ),
\end{equation}
where $c_1(\mathfrak{s}) = c_1(\mathcal{L})$ is the Chern class of the determinant line bundle, $\chi(X)$ is the Euler characteristic of $X$, and $\operatorname{sign}(X) = b^{+}_{2}(X) - b^{-}_{2}(X)$ is the signature of $X$. 
\smallskip%

We now define a mod 2 version of the Seiberg-Witten invariant for spin$^\cc$ structures $\fr{s}$ with $d(\fr{s}) = 0$. When $g$ and $\eta$ are generic, $M_{\fr{s}}$ is zero-dimensional, consisting of a finite number of points. We define the Seiberg-Witten invariant, $\operatorname{SW}(\fr{s})$, as the number of points in $M_{\fr{s}}$, counted modulo 2:
\[
\operatorname{SW}(\fr{s}) = \operatorname{\#} M_{\fr{s}}\,\operatorname{mod}2.
\]
For $b^{+}_{2}(X) > 1$, \cite{K-M} proves that $\operatorname{SW}(\fr{s})$ is independent of the generic choices of $g$ and $\eta$. 
\smallskip%

Assume $X$ is also a symplectic manifold with a symplectic form $\omega$ and an $\omega$-compatible almost-complex structure $J$ such that $g(\cdot, \cdot) = \omega(\cdot, J\cdot)$. As shown in Ch.\,7 of \cite{wild}, every spin$^\cc$ structure on $X$ can be expressed as:
\begin{equation}\label{spin-eps}
W^{+} = L_{\varepsilon} \oplus \left( \Lambda^{0,2} \otimes L_{\varepsilon} \right),\quad 
W^{-} = \Lambda^{0,1} \otimes L_{\varepsilon},
\end{equation}
where $L_{\varepsilon}$ is the line bundle on $X$ with $c_1(L_{\varepsilon}) = \varepsilon \in H^2(X;\zz)$. We denote the spin$^\cc$ structure defined by \eqref{spin-eps} as $\fr{s}_{\varepsilon}$. Letting $K^{*}_X$ denote 
the anticanonical bundle of $X$, the 
determinant line bundle of $\fr{s}_{\varepsilon}$ is given 
by $\mathcal{L} = K^{*}_X \otimes L_{\varepsilon}^2$. 
Consequently, $c_1(\fr{s}_{\varepsilon}) = c_1(\fr{s}_{0}) + 2 \varepsilon$. In terms of $\varepsilon$, the dimension $d(\fr{s}_{\varepsilon})$ becomes:
\[
d(\fr{s}_{\varepsilon}) = c_1(X) \cdot \varepsilon + \varepsilon \cdot \varepsilon.
\]
This formula for \(d(\mathfrak{s}_{\varepsilon})\) is derived from \eqref{dim-1} using the Hirzebruch signature formula (also proved by Rokhlin, \cite{GM}),
\[
c_1(X)^2 = 3 \, \operatorname{sign}(X) + 2 \chi(X),
\]
and noting that \(c_1(\mathfrak{s}_0) = c_1(X)\).
\smallskip

Setting \(\varepsilon = 0\) gives \(d(\mathfrak{s}_{0}) = 0\), and similarly for \(\varepsilon = -c_1(X)\). A key result for our work is the following theorem by Taubes:
\begin{theorem}[Taubes, \cite{Taub-2}]\label{taubes}
Let $(X, \omega)$ be a closed symplectic manifold with $b_{2}^{+}(X) > 1$. Then:
\begin{enumerate}
    \item 
    $\operatorname{SW}(\fr{s}_{0}) = \operatorname{SW}(\fr{s}_{-c_1(X)}) = 1$.
    \smallskip%
    
    \item If $\operatorname{SW}(\fr{s}_{\varepsilon}) \neq 0$, then $\varepsilon$ must satisfy the inequality:
    \[
    \varepsilon \cdot [\omega] \geq 0,
    \]
with equality allowed only for 
$\varepsilon = 0$.
\end{enumerate}
Consequently, if $\varepsilon$ is a non-trivial 2-torsion class, then $\operatorname{SW}(\fr{s}_{\varepsilon}) = 0$.
\end{theorem}
\begin{proof}
For detailed proofs, consult the end-notes of Ch.\,10 in \cite{wild} or Theorem 3.3.29 in \cite{Nic}. \qed
\end{proof}
\medskip%

Let us additionally assume that $X$ satisfies the following topological conditions:
\begin{equation}\label{homology-K3}
b_1(X) = 0,\quad b^{+}_2(X) = 3,\quad b^{-}_2(X) = 19.
\end{equation}
We refer to such a manifold $X$ as a rational homology K3. This terminology arises because these Betti numbers match those of a K3 surface when considered with rational coefficients. The second key result for our analysis, due to Morgan and Szab{\'o}, can be found in \cite{Morg-Sz}, specifically in Theorem 1.1 and Remark 2.2.
\begin{theorem}[Morgan-Szab{\'o}, \cite{Morg-Sz}]\label{morgan}
Let $X$ be a spin 4-manifold and a rational homology K3. Then 
any spin$^\cc$ structure with trivial determinant line bundle has non-trivial Seiberg-Witten invariant.
\end{theorem}
\begin{proof}
The proof is provided in \cite{Morg-Sz}. 
While the original theorem assumes $X$ to be a homotopy K3 (i.e., simply-connected), the proof holds for the general case as well. Besides \eqref{homology-K3}, the only required assumption is that the determinant line bundle of the spin$^\cc$ structure is trivial. \qed
\end{proof}
\smallskip%

Finally, we show how Taubes’ theorem implies that a symplectic manifold \((X,\omega)\) diffeomorphic to a K3 surface has \( c_1(X) = 0 \). Suppose, for contradiction, that \( c_1(X) \neq 0 \). Then \( c_1(X) \) cannot be a 2-torsion class, as no such 
classes exist in a K3 surface. By Taubes’ theorem, we have:
\[
\operatorname{SW}(\fr{s}_{0}) = 1,\quad \operatorname{SW}(\fr{s}_{-c_1(X)}) = 1.
\]
Here, the Chern class of \( \fr{s}_0 \) is \( c_1(X) \), and that of \( \fr{s}_{-c_1(X)} \) is \( -c_1(X) \). Given the absence of 2-torsion classes, \( X \) must then admit at least two distinct spin\(^\cc\) structures with non-zero Seiberg-Witten invariants. However, if \( X \) is diffeomorphic to a K3 surface, it can have only one such structure. For a proof, see the end notes of Ch.\,10 in \cite{wild}.

\section{Rokhlin and Viro numbers} 
The following material on the Viro index is well-known (see, e.g., \cite{Viro, Nets, Nemir-1}). Let $(X,J)$ be a simply-connected almost-complex 4-manifold, and let $L \subset X$ be an embedded totally real torus. We assume that the homology class of $L$ is zero modulo 2 in $X$, i.e., $[L] = 0 \in H_2(X; \mathbb{Z}_2)$.
\smallskip%
 
Consider a simple closed loop $\gamma$ on $L$, a non-vanishing vector field $\dot{\gamma}$ tangent to $\gamma$, and a non-vanishing normal vector field $\nu$ to $\gamma$ on $L$. Here, \say{normal} means $\nu$ is everywhere transverse to $\dot{\gamma}$. By pushing off $\gamma$ in the direction of the field $J \nu$, we obtain a loop $\gamma^{\#} \subset X - L$. The isotopy class of $\gamma^{\#}$ in $X - L$ is well-defined because any two non-vanishing normal vector fields on $\gamma$ are homotopic through non-vanishing vector fields on the tangent bundle of $L$ restricted to $\gamma$.
\smallskip%

The Viro index of $\gamma$ is defined as the linking number modulo 2:
\[
V(\gamma) = \operatorname{lk}(\gamma^{\#}, L) \in \zz_2.
\]
The linking number is defined as follows. Since $X$ is simply-connected, there exists an immersed disk $D$ that bounds $\gamma$ in $X$. We choose $D$ to be transverse to $L$. Then, the linking number is the number of intersection points between $D$ and $L$, modulo 2:
\[
\operatorname{lk}(\gamma^{\#}, L) = 
\operatorname{\#} (L \cap D)\,\operatorname{mod}2.
\]
If $D'$ is another disk bounding $\gamma$, then 
the union of $D$ and $D'$ forms an immersed sphere in $X$ denoted as $S$. The difference between the number of intersection points of $L$ with $D$ and $D'$ (modulo 2) is equal to the number of intersection points between $S$ and $L$ (also modulo 2). Since $L$ is nullhomologous modulo 2, the latter number is always even. Therefore, the Viro index is well-defined for vanishing homology classes $[L] \in H_2(X; \mathbb{Z}_2)$.
\smallskip%

We let $\alpha \in H^1(L; \mathbb{Z})$ to be half of the Maslov class of $L$.
\begin{lemma}\label{gamma-is}
Let $(X,J)$ be a simply-connected almost-complex 4-manifold homotopy-equivalent to a complex K3 surface, and suppose $c_1(X) = 0$. If $L$ is nullhomologous modulo 2 and the modulo 2 reduction of $\alpha$ vanishes, then there exists 
a simple loop $\gamma$ on $L$ such that $V(\gamma) = 1$. 
\end{lemma}
The proof occupies the present section. 
\smallskip%

$\Lambda^{2}_{\cc}(T_{X})|_{L}$ has a preferred non-vanishing section, which comes from a non-vanishing section of $\Lambda^{2}_{\rr}(T_{L})$ through the isomorphism \eqref{iso}. Let $\sigma$ be a generic extension of this section over the entirety of $\Lambda^{2}_{\cc}(T_{X})$. Define:
\[
\Sigma = \left\{ x \in X\ |\ \sigma(x) = 0 \right\}.
\]
If not empty, $\Sigma$ is a smooth, oriented, embedded surface in $X$ disjoint from $L$. Suppose that $\gamma$ bounds
an embedded (not necessarily orientable) surface $M \subset X$ such that $M$ is nowhere-tangent to $L$ along $\gamma$ and such that the interior of $M$ is transverse to both $L$ and $\Sigma$. We call $M$ a membrane for the loop $\gamma$. Such a surface always exists, since $X$ 
is simply-connected. 
\smallskip%

Let $N_M = T_X/T_M|_M$ be the normal bundle to $M$ in $X$. We define its modulo 2 Euler class, $d(M) \in \mathbb{Z}_2$, relative to the chosen normal vector field $\nu$ along $\gamma = \partial M$. We extend $\nu$ from a non-vanishing section over $\gamma$ to a generic section of $N_M$ over $M$. $d(M)$ is the number of zeros of this extension modulo 2, independent of the extension choice.
\smallskip%

Denote by $\#(M \cap \Sigma)$ and $\#(M \cap L)$ the intersection numbers (interior) of $M$ with $\Sigma$ and $L$, respectively. Define $R(M, \sigma) \in \zz_2$ (Rokhlin index of $M$) and $q(\gamma) \in \zz_2$ as:
\[
R(M, \sigma) = d(M) + \#(M \cap \Sigma) + \#(M \cap L)\,\operatorname{mod}2,\quad
q(\gamma) = d(M) + \#(M \cap L)\,\operatorname{mod}2
\]
Nemirovski's theorem \cite{Nemir-1} relates the Rokhlin index, the Viro index, and a membrane.
\begin{theorem}[Nemirovski, \cite{Nemir-1}]
Let $M$ be a membrane for $\gamma$ on $L$, chosen so the tangent half-space to $M$ along $\gamma$ is spanned by $\dot{\gamma}$ and $J\nu$. Then, with this choice of $M$:  
\begin{equation}\label{nemirovski}
  R(M, \sigma) = 1 + V(\gamma)\ \text{\normalfont{mod}}\ 2.  
\end{equation}
\end{theorem}
\begin{proof}
See Lemma 1.13 in \cite{Nemir-1}. \qed
\end{proof}
\smallskip%

Observe that $\#(M \cap \Sigma)$ coincides with $\alpha(\gamma)$, calculated modulo 2. Therefore, if the modulo 2 reduction of $\alpha$ vanishes, we have the following relationship:
\begin{equation}\label{Rq}
R(M, \sigma) = q(\gamma).
\end{equation}
As pointed out in \cite{Rokh-1}, if $L$ is a characteristic surface (i.e., nullhomologous modulo 2), then $q(\gamma)$ only depends on the homology class of the loop $\gamma$ in $H_1(L; \mathbb{Z}_2)$ and is independent of the specific choice of membrane $M$. Furthermore, $q$ is a quadratic function over $H_1(L; \mathbb{Z}_2)$ with an associated Arf invariant. This Arf invariant can be calculated as 
follows: If $e_1$ and $e_2$ form any basis for $H_1(L; \mathbb{Z}_2)$, then the Arf invariant of $q$, denoted by $\operatorname{Arf}(q)$, is defined as:
\begin{equation}\label{arf-def}
\operatorname{Arf}(q) = q(e_1) q(e_2).
\end{equation}
A detailed discussion of the function $q$ and the associated Arf invariant can be found in \cite{Mat, GM}, as well as in the end-notes of Ch.\,11 in \cite{wild}.
\smallskip%

Rokhlin and Freedman-Kirby established the following congruence:
\begin{equation}\label{arf}
\mathrm{sign}(X) - [L]^2 = 8\,\text{Arf}(q)\,\operatorname{mod} 16.
\end{equation}
A detailed proof of this congruence, along with its background and applications, can be found in \cite{Mat, GM}. If $X$ is homotopy-equivalent to a K3 surface, then $\mathrm{sign}(X) = -16$. Further, if $L$ is totally real, then $[L]^2 = 0$. Plugging these into \eqref{arf} yields $\operatorname{Arf}(q) = 0$. It follows from \eqref{arf-def} that there must be a simple loop $\gamma$ on $L$ with $q(\gamma) = 0$. Lemma \ref{gamma-is} now follows from \eqref{nemirovski} and \eqref{Rq}.

\begin{remark}
Congruence \eqref{arf} admits a generalization to the non-simply-connected case, as discussed in \S\,2.6 in \cite{Deg-It-Kh}. Combining with Theorem A in \cite{Rub-Strle}, it leads to a generalization of Theorem \ref{t:A} applicable to symplectic four-tori (Corollary 9.2 in \cite{Ab-Smith}). However, we will not study this generalization further.
\end{remark}

\section{Luttinger surgery}\label{luttinger}
We assume familiarity with Luttinger surgery (see, e.g., \cite{Aur-Donald-Katz, Eliash-Polt, Mish, Nemir-1, Nemir-2}). 

Let $D$ be a small disk in $\mathbb{R}^2$ with standard Lagrangian torus bundle $\pi \colon D\times T^2 \rightarrow D$ and symplectic form induced by $(T^{*}_{T^{2}}, \omega_0)$. 
Suitable local coordinates $(x,y)$ on $D$ and $(\theta_x, \theta_y)$ on $T^2$ express $\omega_0$ as $d\theta_x \wedge dx + d\theta_y \wedge dy$. Using polar coordinates $(r, \varphi)$ on $D$ with $(x, y) = (r \cos{\varphi}, r \sin{\varphi})$, define a multi-valued function $f_{m,n}(r,\varphi) = m  x \varphi + n y \varphi$, $(m, n) \in \zz^2$, on $N = D - {0}$. Though $f_{m,n}$ is not well-defined on $N$, the symplectomorphism $\psi \colon \pi^{-1}(N) \to \pi^{-1}(N)$ generated by the time-1 Hamiltonian flow of $f_{m,n} \circ \pi$ is well-defined. Denote the $k$th power of $\psi$ as $\psi^k$. 
\smallskip%

Let $L$ be a Lagrangian torus in $(X, \omega)$. A neighborhood $U$ of $L$ is symplectically identified with $\pi^{-1}(D)$ such that $L$ corresponds to $\pi^{-1}(0)$. Choose a basis for $H_1(T^2; \mathbb{Z})$ and a simple loop $\gamma$ on $L$, represented by $(m,n) \in \mathbb{Z}^2$. (Note that if \(\gamma\) is a simple loop, then either \((m, n) = (0, 0)\) or \(m, n\) are coprime.) Luttinger surgery on $X$ along $L$ with respect to $\gamma$ and $k$ (denoted $X(L,\gamma,k)$) is obtained by removing $U$ from $X$ and gluing it back along $U - L$ via $\psi^k$:
\[
X(L,\gamma,k) = 
\left( (X - L) \cup U \right) /\sim,\quad 
\text{$x_1 \in X - L \sim x_2 \in U$ iff $x_1, x_2 \in U - L$ and $\psi^{k}(x_1)= x_2$.}
\]
Since $\psi^k$ is a symplectomorphism, $X(L,\gamma,k)$ carries a natural symplectic form, whose deformation class is independent of the construction choices \cite{Aur-Donald-Katz}. We refer to \cite{Mish}, Example 2, for a description of Luttinger surgery using $f_{m,n}$.
\smallskip%

Let us analyze the effect of Luttinger surgery on $H^1$ and $H^2$, specializing to (homotopy) K3 surfaces. For abbreviation, we denote $X(L,\gamma,k)$ as $\tilde{X}$.
\begin{lemma}\label{b1-null}
If $\pi_1(X)$ is trivial, $L$ is nullhomologous modulo 2 in $X$, and 
$\gamma$ is such that $V(\gamma) = 1$, then there exists $k \in \zz$ such that 
$b_1(\tilde{X}) = 0$ and such that $H^2(\tilde{X};\zz)$ has non-trivial 2-torsion. 
\end{lemma}
\begin{proof}
Since $\gamma$ is a simple loop representing a primitive class in $H_1(L; \mathbb{Z})$, there exists a simple loop $\beta$ such that their classes form a basis for $H_1(L; \mathbb{Z})$. In $U - L$, pick a loop $\mu$ such that there exists a nullhomotopy of $\mu$ in $U$ intersecting $L$ at a 
single point. We call such a loop a meridian for future reference.  $H_1(U - L; \mathbb{Z}) = \mathbb{Z}^3$ is generated by $\mu, \gamma^{\#}, \beta^{\#}$.
\smallskip%

We consider two cases: (1) nullhomologous $L$ and (2) homologically non-trivial but nullhomologous modulo 2 $L$. We only prove (1), as the argument for (2) is similar.
\smallskip%

We have the following piece of the Mayer–Vietoris sequence:
\begin{equation}\label{M-V-X}
\ldots \to H_1(U - L) \to H_1(X - L) \oplus H_1(U) \to H_1(X) \to 0,
\end{equation}
and a similar sequence for $\tilde{X} - L$ and 
$\tilde{X}$:
\begin{equation}\label{M-V-Xt}
\ldots \to H_1(U - L) \to H_1(\tilde{X} - L) \oplus H_1(U) \to H_1(\tilde{X}) \to 0.
\end{equation}
Since $U - L \subset X - L$, $\mu$ can be considered a loop in $X - L$. Let us show that if $L$ is nullhomologous, then $H_1(X - L; \mathbb{Z}) = \mathbb{Z}$, generated by $\mu$. Any loop $\delta$ in $X - L$ bounds an immersed disk $C$ in $X$. By arranging the intersections of $C$ and $L$ to be transverse and removing small neighborhoods of the intersection points, we obtain a sphere with holes $C' \subset C$ that can be isotoped into $X - L$. The boundary of $C'$ consists of $\delta$ and several meridians homologous to $\mu$, implying $\delta$ is homologous to a multiple of $\mu$. To show $\mu$ is non-trivial in $X - L$, assume it is nullhomologous. Then, a surface $C$ in $X - L$ bounding $\mu$, when joined with a nullhomotopy of $\mu$ in $U$, would intersect $L$ once. This is impossible since $L$ is nullhomologous.
\smallskip%

If we choose $\mu, \gamma^{\#}, \beta^{\#}$ as a basis for $H_1(U - L)$, $\mu$ as a generator for $H_1(X - L)$, and $\gamma, \beta$ as a basis for $H_1(U)$, then the homomorphism $H_1(U - L) \to H_1(X - L) \oplus H_1(U)$ in \eqref{M-V-X} is the identity matrix. The homomorphism $H_1(U - L) \to H_1(\tilde{X} - L) \oplus H_1(U)$ in \eqref{M-V-Xt} is given by the matrix:
\begin{equation}\label{pq}
\begin{pmatrix}
    1 &   p &   q  \\
    k &   1 &   0  \\
    0 &   0 &   1
\end{pmatrix},\quad 
\text{where $p = V(\gamma)\operatorname{mod} 2$ and $q = V(\beta)\operatorname{mod} 2$.}
\end{equation}
The rank of $H_1(\tilde{X}; \mathbb{Z})$ is the corank of this matrix. Choose $k$ such that the above matrix has non-vanishing determinant. On the other hand, (the absolute value of) that determinant is exactly the order of $H_1(\tilde{X};\zz)$. Choosing $k$ odd, we arrange so that the order of $H_1(\tilde{X};\zz)$ is even. This can be done 
because $p = 1\operatorname{mod} 2$. But then 
there must be elements of order 2 in $H_1(\tilde{X};\zz)$, and hence in $H^2(\tilde{X};\zz)$. \qed
\end{proof}
\smallskip%

\begin{lemma}\label{b1-nullagain}
If $\pi_1(X)$ is trivial and $L$ has non-trivial homology class, then $b_1(\tilde{X}) = 0$.
\end{lemma}
\begin{proof}
We analyze the Mayer-Vietoris sequence \eqref{M-V-Xt} with rational coefficients. To show $b_1(\tilde{X}) = 0$, it suffices to prove the surjectivity of $H_1(U - L) \to H_1(\tilde{X} - L) \oplus H_1(U)$. From the proof of Lemma \ref{b1-null} we 
establish that $H_1(\tilde{X} - L)$ is generated by $\mu$. Since $L$ has non-trivial homology, \Poincare duality guarantees an immersed surface $C$ with non-zero intersection number with $L$. By arranging $C$ to be transverse to $L$ and removing small neighborhoods of intersections, we obtain a sphere with holes $C' \subset C$ whose boundary consists of meridians homologous to $\mu$. This implies that a multiple of 
$\mu$ is nullhomologous in 
$\tilde{X} - L$. Since the Mayer-Vietoris sequence \eqref{M-V-Xt} maps $\gamma^{\#}, \beta^{\#}$ to $\gamma, \beta$ in $H_1(U)$, it follows that $H_1(U - L)$ surjects onto $H_1(U)$.\qed
\end{proof}
\smallskip%

Let us examine the effect of Luttinger surgery on the Chern class and Euler characteristic.
\begin{lemma}\label{euler}
$\chi(X) = \chi(\tilde{X})$. 
\end{lemma}
\begin{proof}
$\chi(X)$ depends only on the ranks of $H_k(U)$, $H_k(X - L)$, $H_k(U - L)$, not on the homomorphisms encountered in the Mayer–Vietoris sequence. Since $\tilde{X} - L$ coincides with 
$X - L$, it follows that 
$H_k(\tilde{X} - L)$ coincide with $H_k(X - L)$. The groups 
$H_k(U)$ are the same for both sequences. Thus, $\chi(\tilde{X}) = \chi(X)$.
\qed
\end{proof}
\smallskip%

\begin{lemma}\label{c1}
If $c_1(X) = 0$, then $
c_1(\tilde{X}) = -k\, \alpha(\gamma)\,[L] \in H^2(\tilde{X};\zz)
$, where $\alpha$ is the Maslov class of $L$ in $X$ 
and $[L]$ is the dual to $L$ in $\tilde{X}$ (i.e., $[L]$ is the Chern class 
of the complex line bundle with divisor $L$.) Note that the orientation of $L$ does not appear in this formula because it implicitly appears in our definition of Luttinger surgery.
\end{lemma}
\begin{proof}
This proof is derived from 
\cite[\S\,2.2]{Aur-Donald-Katz}. 
Let $\operatorname{ker} d\pi$ denote the vertical tangent bundle for 
$\pi \colon U \to D^2$. Since $\pi \circ \psi = \pi$, it follows that $\psi$ induces a bundle 
morphism $\psi_* \colon \Lambda^2_{\rr}( \operatorname{ker} d\pi ) \to \Lambda^2_{\rr}( \operatorname{ker} d\pi )$. Using the description of $\psi$ in terms of $f_{m,n}$, one finds a non-vanishing section $\sigma$ of $\Lambda^2_{\rr}( \operatorname{ker} d\pi )$ such that $\psi_{*} \sigma = \sigma$. Explicitly, if we let $v_{x}$ and $v_y$ be the Hamiltonian 
vector fields for the functions $x$ and $y$, respectively, then we can set 
$\sigma = v_{x} \wedge v_y$. Both $v_{x}$ and $v_y$ belong to $\operatorname{ker} d\pi$. Moreover, since $\psi$ preserves $x$ and $y$ and is a symplectomorphism, it also preserves $v_x$ and $v_y$.
\smallskip%

Choose an almost-complex structure $J$ on $X$. Let $s$ be a 
section of $\Lambda^2_{\cc} (T_{X})$ over the whole of $X$. Since $\operatorname{ker} d\pi$ is an orientable Lagrangian plane field over $U$, we have a canonical isomorphism 
\begin{equation}\label{ker}
    \Lambda^2_{\rr}( \operatorname{ker} d\pi ) \otimes \cc \cong 
\Lambda^2_{\cc} (T_{X})|_{U}.
\end{equation}
By \eqref{ker}, $\sigma$ induces a section of $\Lambda^2_{\cc} (T_{X})|_{U}$, which we shall also denote by $\sigma$. Then, over $U - L \subset X$, we have $s = g \cdot \sigma$, where $g \colon U - L \to \cc^{*}$. Letting $\delta \in H^1(U - L;\zz)$ be the cohomology class of the mapping $g$, we have $\delta(\gamma^{\#}) = \alpha(\gamma)$, 
$\delta(\beta^{\#}) = \alpha(\beta)$, and $\delta (\mu) = 0$. Here $\mu, \gamma^{\#}, \beta^{\#}$ are defined as in the proof of Lemma \ref{b1-null}.
\smallskip%

Using the identification $\tilde{X} - L \cong X - L$, we endow 
$\tilde{X} - L$ with the almost-complex structure $J$. Now, consider 
the subset $U - L \subset \tilde{X}$. Let $\psi^{k}_{*} s$ represent 
a section of $\Lambda^2_{\cc} (T_{X})|_{U-L}$, understood with respect to the structure $\psi^{k}_{*}J$. This section is obtained through the pushforward of the section $s$ under $\psi^{k}$, and it satisfies the relation:
\[
\psi^{k}_{*} s = (g \circ \psi^{-k}) \cdot \sigma.
\]
In this formula, $\sigma$ is the section of $\Lambda^2_{\cc} (T_{X})|_{U-L}$, obtained using the almost-complex structure $\psi^{k}_{*} J$ and the isomorphism \eqref{ker}. Additionally, $(g \circ \psi^{-k})$ denotes the pushforward of the function $g$ under $\psi^k$.
\smallskip%

Observe that $\psi^{k}_{*} J$ does not extend to cover the entirety of $U$. 
To fix this, let us consider a tubular neighborhood $V \subset U$ of $L$, which is strictly smaller than $U$. Choose an almost-complex structure $J'$ over the entire region $U$, such that $J' = \psi^{k}_{*} J$ holds on $U - V$. Let $s'$ be a section of $\Lambda^2_{\cc} (T_{X})|_{U-L}$, understood with respect to the structure $J'$, that is chosen in such a way that $s = s'$ on $U-V$. (Since $U$ and $V$ are homotopy-equivalent, it follows that such an extension of $s|_{U - V}$ exists and is unique up to homotopy.) Now, let $\sigma'$ be the section of $\Lambda^2_{\cc} (T_{X})|_{U}$ 
obtained using the almost-complex structure $J'$ and the isomorphism \eqref{ker}. Then, over $U - L \subset \tilde{X}$, we have 
$s' = g' \cdot \sigma'$. The homotopy class of $g'$ is equal to that of $g \circ \psi^{-k}$ since $g \circ \psi^{-k}= g'|_{U-V}$. Letting $\delta' \in H^1(U - L;\zz)$ be the cohomology class of the mapping $g'$, we have $\delta'(\gamma^{\#}) = \alpha(\gamma)$, 
$\delta'(\beta^{\#}) = \alpha(\beta)$, and $\delta' (\mu) = -k\, \alpha(\gamma)$.
\smallskip%

Therefore, over $U - L \subset \tilde{X}$, we have $s' = (g'/g) \cdot (g \cdot \sigma')$, where $\sigma'$ is a section of $\Lambda^2_{\cc} (T_{X})$ over $U$, and $s'$ is a section of $\Lambda^2_{\cc} (T_{X})$ over $\tilde{X} - L$. The class of the mapping $g'/g$ becomes equal to $\delta' - \delta \in H^1(U - L;\zz)$, and it satisfies $[\delta' - \delta](\mu) = -k\, \alpha(\gamma)$ and 
$[\delta' - \delta](\gamma^{\#}) = [\delta' - \delta](\beta^{\#}) = 0$. This completes the proof. 
\qed
\end{proof}
\smallskip%

\begin{lemma}\label{sigma}
$\operatorname{sign}(X) = \operatorname{sign}(\tilde{X})$. 
\end{lemma}
\begin{proof}
Novikov's additivity theorem implies that if we glue together two compact oriented 4-manifolds along a connected component of their boundaries, the signature of the resulting manifold does not depend on the choice of the gluing map. For the proof, see the end-notes of Ch.\,4 in \cite{wild}. \qed
\end{proof}
\medskip%

Luttinger surgery can be extended to Lagrangian Klein bottles. Define $h \colon U \to U$ as:
\[
h(\theta_x, \theta_y, x, y) = 
(-\theta_x + \pi, \theta_y, -x, y).
\]
The map $h$ has the following properties: (1) $h$ is a fixed-point free involution; (2) $h^* \omega_0 = \omega_0$; (3) $h^*x = -x$, $h^*y = y$, and $h$ maps the torus $L$ to itself. Define $K$ and $U_{K}$ as:
\[
K = L/\sim,\quad 
\text{$x_1 \in L \sim x_2 \in L$ iff $h(x_1)= x_2$;}
\]
\[
U_{K} = U/\sim,\quad 
\text{$x_1 \in U \sim x_2 \in U$ iff $h(x_1)= x_2$.}
\]
$K$ is the resulting Lagrangian Klein bottle, and $U_{K}$ is symplectomorphic to a tubular neighbourhood of an embedded Lagrangian Klein bottle in a symplectic 4-manifold.
\smallskip%

Setting $m = 0$ and $n = 1$ in $f_{m, n}$, we let $\psi: U \to U$ be the associated symplectomorphism. Since $h$ commutes with $f_{0, 1}$, there exists a symplectomorphism $\psi_K$ such that the following diagram commutes:
\[
\begin{tikzcd}
U \arrow{d}[swap]{/h} \arrow{r}{\psi} & U \arrow{d}{/h}\\
U_{K} \arrow{r}{\psi_{K}} & U_{K}.
\end{tikzcd}
\]
Here, the vertical arrows identify $x$ and $h(x)$. If $L$ is an embedded Lagrangian Klein bottle in a symplectic manifold $(X, \omega)$, then a tubular neighborhood of $K$ is symplectomorphic to $U_K$. We define $X(K)$ as the symplectic manifold obtained by performing Luttinger surgery on $X$ with respect to $K$. This involes removing $U_{K}$ from $X$ and gluing it back along $U_{K} - K$ via $\psi_{K}$:
\[
X(K) = \left( (X - K) \cup U_{K} \right) /\sim,\quad 
\text{$x_1 \in X - K \sim x_2 \in U_{K}$ iff $x_1, x_2 \in U_{K} - K$ and $\psi_K(x_1) = x_2$.}
\]
Since $\psi_K$ is a symplectomorphism, $X(K)$ carries a natural symplectic form.
\smallskip%

Both Lemma \ref{euler} and Lemma \ref{sigma} hold for Luttinger surgery on a Klein bottle, and will be used in the proof of Theorem \ref{klein_remain} below. Lemma \ref{euler} only relies on the algebraic definition of the Euler characteristic, and Lemma \ref{sigma} on Novikov's theorem. 

\section{Proof of Theorem \ref{t:A}}\label{proof} 
Let us first prove (1). We proceed by contradiction. Assume $L$ is not a Maslov-zero torus but has a non-trivial homology class. Since $L$ is not Maslov-zero, there exists a simple loop $\gamma$ on $L$ with $\alpha(\gamma) \neq 0$. We perform Luttinger surgery on $X$ with respect to $L$, $\gamma$, and $k = 1$ (denoted by $\tilde{X} = X(L,\gamma,1)$). It follows from Lemma \ref{c1} that 
\[ 
c_1(\tilde{X}) = -\alpha(\gamma)\,[L]. 
\]
Let us show that \(\tilde{X}\) is a rational homology K3. Lemma \ref{b1-nullagain} implies $b_1(\tilde{X}) = 0$ since $L$ has a non-trivial homology class. Since $\tilde{X}$ is closed, $b_1(\tilde{X}) = 0$ implies $b_3(\tilde{X}) = 0$. Using Lemma \ref{euler}, we get $\chi(X) = \chi(\tilde{X})$, leading to $2 + b_2(X) = 2 + b_2(\tilde{X})$. This implies the rational Betti numbers of $\tilde{X}$ match those of $X$. Lemma \ref{sigma} then gives $b^{+}_{2}(\tilde{X}) = b^{+}_{2}(X)$ and $b^{-}_{2}(\tilde{X}) = b^{-}_{2}(X)$.
\smallskip%

Let us show that $c_1(\tilde{X}) = 0$. Taubes' theorem states that 
\[
\operatorname{SW}(\fr{s}_{-c_1(X)}) = 1.
\] 
Since \( L \) is Lagrangian, \( c_1(\tilde{X}) \cdot [\tilde{\omega}] = 0 \). Therefore, by Taubes' theorem, \( c_1(\tilde{X}) = 0 \). 
\smallskip%

Since $c_1(\tilde{X}) = 0$, $\tilde{X}$ is spin. Hence, the intersection form of \(\tilde{X}\) must be even and isomorphic to \( 3H \oplus (-2E_8) \) (see Theorem 5.3 in \cite{Milnor1973}). Since \( c_1(\tilde{X}) = 0 \), it follows that \([L] \in H^2(\tilde{X};\mathbb{Z})\) is a torsion class. Using the three copies of \( H \), we can find a triplet of surfaces \( C_1, C_2, C_3 \) within \(\tilde{X}\), each with a positive self-intersection number, and such that they are pairwise disjoint and disjoint from \( L \). Given the natural identification between \(\tilde{X} - L\) and \(X - L\), and the fact that \( C_i \subset \tilde{X} - L \), we can also view the surfaces \( C_i \) as surfaces in \(X\). The following properties of \([C_i]\) and \([L]\), as elements in \( H_2(X;\mathbb{Z}) \), are immediate:
\[ [L]^2 = 0, \quad [C_i]^2 > 0, \quad [C_i] \cdot [L] = 0 \ \text{for } i = 1,2,3, \quad \text{and} \quad [C_i] \cdot [C_j] = 0 \ \text{whenever } i \neq j. \]
However, such a configuration of non-zero cycles cannot be realized in a space with signature \((3,19)\).
\smallskip%

Let us now prove (2). We proceed by contradiction again. Assume $L$ is nullhomologous modulo 2, but the modulo 2 reduction of $\alpha$ vanishes. From Lemma \ref{b1-null}, there exists a Luttinger surgery result $\tilde{X} = X(L,\gamma,k)$ such that $b_1(\tilde{X}) = 0$ and $H^2(\tilde{X};\mathbb{Z})$ has 2-torsion. Similar to part (1), Lemma \ref{c1} gives $c_1(\tilde{X}) = -k \alpha(\gamma) [L]$, and arguments like before establish that $\tilde{X}$ is a rational homology K3 with $c_1(\tilde{X}) = 0$.
\smallskip%

Let $\varepsilon \in H^2(\tilde{X};\mathbb{Z})$ be a non-trivial 2-torsion class. Since $c_1(\tilde{X}) = 0$ and $\varepsilon$ has order 2, the spin$^\cc$ structure $\mathfrak{s}_{\varepsilon}$ has a trivial determinant line bundle. Using the Morgan-Szab{\'o} theorem, we conclude that  
\[ \operatorname{SW}(\mathfrak{s}_{\varepsilon}) \neq 0. \]
However, this contradicts Taubes' theorem. Since $\varepsilon$ has order 2, it follows that $\varepsilon \cdot [\omega] = 0$.  Taubes' theorem states that such $\varepsilon$ must be zero.

\section{On Lagrangian Klein bottles}\label{klein}

Congruence \eqref{arf} admits a generalization to unoriented characteristic surfaces. See the papers \cite{GM} and \cite{Mat}, both in the same volume, 
edited by Guillou and Marin. Combining with the result of 
Morgan-Szab{\'o}, it leads to the following Klein bottle version of Theorem \ref{t:A}: 
\begin{theorem}\label{klein_remain}
Let $(X,\omega)$ be a symplectic 
4-manifold homotopy-equivalent to a complex K3 surface, and suppose $c_1(X) = 0$. An embedded Lagrangian Klein bottle $K \subset (X,\omega)$, if exists, must have 
non-trivial homology class.
\end{theorem}
\begin{proof}
Let us show that if $\tilde{X}$ is obtained 
from $X$ via Luttinger surgery along a Lagrangian Klein bottle $K$, regardless of whether $K$ is nullhomologous or not, then $\tilde{X}$ is a rational homology K3. 
\smallskip

To this end, let us first show that $b_1(\tilde{X}) = 0$. Consider a loop $\delta$ in $\tilde{X}$. Perturb it if necessary to ensure it is disjoint from $K$. Denote $2\delta$ as a loop homotopic to a double of $\delta$. We will show that $2\delta$ is homologically trivial in $\tilde{X} - K$ (equivalent to $X - K$). Since $X$ is simply-connected, there exists an orientable immersed surface $C \subset X$ such that $\partial C = 2 \delta$. Arrange $C$ to be transverse to $K$ in $X$. This ensures an even number of intersection points between $C$ and $K$. Using Gompf's argument (Lemma 4.10 in \cite{B-L-W}), we can replace $C$ with another orientable surface $C'$ such that $\partial C' = 2\delta$ but $C'$ is disjoint from $K$. This modification shows that $2 \delta$ is nullhomologous in $X - K$, hence in $\tilde{X}$. We have established that $b_1(\tilde{X}) = b_1(X)$.
\smallskip%

Lemma \ref{euler} applies to Luttinger surgery with Klein bottles, giving $b_2(\tilde{X}) = b_2(X)$. Lemma \ref{sigma} also applies to Luttinger surgery with Klein bottles, resulting in $b^{+}_2(\tilde{X}) = b^{+}_2(X)$ and $b^{-}_2(\tilde{X}) = b^{-}_2(X)$. We have established that $\tilde{X}$ is a rational homology K3.
\smallskip%

Let us show that $c_1(\tilde{X}) = 0$. Since $c_1(X) = 0$, then over $\tilde{X} - K$, $\Lambda^2_{\cc} (T_{\tilde{X}})$ 
has a nowhere-vanishing section $s$. $K$ is a homotopy retract of $U_K$, and a complex line bundle over a non-orientable surface like $K$ must be trivial. Therefore, $\Lambda^2_{\cc} (T_{\tilde{X}})$ is trivial over $U_K$ and is trivialised over $\partial U_K$ by $s$. The obstruction to extending $s$ from $\partial U_K$ to the entire $U_K$ lies in $H^1(U_K, \partial U_K)$. Let us show that this group is trivial. To begin with, consider the long exact sequence for $(U_K, \partial U_K)$:
\[
\ldots \longrightarrow 
H^{0}(U_K) \longrightarrow  
H^{0}(\partial U_K) \longrightarrow
H^1(U_K, \partial U_K) \longrightarrow
H^{1}(U_K) \longrightarrow 
H^{1}(\partial U_K) \longrightarrow 
\ldots
\]
Here, the coefficients are to be taken in \( \mathbb{Z} \). It suffices to show that \( H^0(U_K) \to H^0(\partial U_K) \) is surjective and \( H^1(U_K) \to H^1(\partial U_K) \) is injective. The first claim follows since \( \partial U_K \) is connected. To prove the second, view \( U_K \) as a normal bundle to \( K \) in \( X \), with \( K \) as the zero section, and introduce another section \( K_1 \) as follows. Recall that $K$ being totally real implies \( T_K \cong U_K \). Since $T_{K}$ admits a non-vanishing section, so does $U_K$. Denote this section as $K_1$, and arrange that it lies entirely in $\partial U_K$. Consider the sequence of maps:
\[
K_1 \longrightarrow \partial U_K \longrightarrow U_K \longrightarrow K,
\]
where $U_K \longrightarrow K$ is a deformation retraction, and the rest are inclusions. Suppose \( \varphi \in H^1(U_K) \) is non-trivial but becomes trivial in \( H^1(\partial U_K) \). Then \( \varphi \) also vanishes when restricted to \( K_1 \). Since \( K_1 \) is isotopic to \( K \) in \( U_K \), it follows that \( \varphi \) must also vanish on \( K \). However, this leads to a contradiction because \( U_K \) retracts onto \( K \), inducing an isomorphism in cohomology. Hence, the map \( H^1(U_K) \to H^1(\partial U_K) \) is indeed injective.
\smallskip%

Corollary 2.3 in \cite{Nemir-1} states that Luttinger surgery along a nullhomologous Lagrangian Klein bottle would result in a $\tilde{X}$ with 
$H_1(\tilde{X};\zz_2) = \zz_2$. However, we have already seen that the free part of $H_1(\tilde{X};\zz)$ 
vanishes. Combining these results implies that $H_1(\tilde{X};\zz)$ must have 2-torsion elements. It follows from the universal coefficient theorem that 
$H^2(\tilde{X};\zz)$ also has 2-torsion elements. We have also established that $c_1(\tilde{X}) = 0$. This leads to a contradiction with the results of Taubes and Morgan-Szab{\'o} in the same way as in the proof of Theorem \ref{t:A}. \qed
\end{proof}

\bibliographystyle{plain}
\bibliography{references}

\end{document}